\documentclass[11pt]{elsarticle} 

\usepackage{amsmath, amsthm, amssymb, amsfonts}

\newtheorem{theorem}{Theorem}[section]
\newtheorem{lemma}[theorem]{Lemma}
\newtheorem{proposition}[theorem]{Proposition}
\newtheorem{corollary}[theorem]{Corollary}

\newtheorem{problem}{Question}

\newenvironment{definition}[1][Definition]{\begin{trivlist}
\item[\hskip \labelsep {\bfseries #1}]}{\end{trivlist}}

\title{The Sarason Sub-Symbol and the Recovery of the Symbol of Densely Defined Toeplitz Operators over the Hardy Space}
\author{Joel A. Rosenfeld}

\begin{document}

\begin{abstract}While the symbol map for the collection of bounded Toeplitz operators is well studied,  there has been little work on a symbol map for densely defined Toeplitz operators. In this work a family of candidate symbols, the Sarason Sub-Symbols, is introduced as a means of reproducing the symbol of a densely defined Toeplitz operator. This leads to a partial answer to a question posed by Donald Sarason in 2008. In the bounded case the Toeplitzness of an operator can be classified in terms of its Sarason Sub-Symbols. This justifies the investigation into the application of the Sarason Sub-Symbols on densely defined operators. It is shown that analytic closed densely defined Toeplitz operators are completely determined by their Sarason Sub-Symbols, and it is shown for a broader class of operators that they extend closed densely defined Toeplitz operators (of multiplication type).\end{abstract}

\maketitle

\section{Introduction}

The study of bounded Toeplitz operators over the Hardy space $H^2(\mathbb{T})$ is a well developed subject where there are several equivalent definitions of a Toeplitz operator. The simplest definition of a bounded Toeplitz operator is an extension of the definition of a Toeplitz matrix. In this case an operator, $T$, is called a Toeplitz operator if the matrix representation of the operator, with respect to the orthonormal basis $\{ e^{in\theta} \}_{n=0}^\infty$ is constant along the diagonals. Algebraically, this can be represented as $S^* T S = T$. Here $S=M_z$ is the shift operator for the Hardy space. If the coeficients corresponding to each diagonal of the matrix are the Fourier coeficients of a function $\phi \in L^\infty(\mathbb{T})$, then $T = P_{H^2(\mathbb{T})} M_\phi$. Here $P_{H^2(\mathbb{T})}$ is the projection from $L^2(\mathbb{T})\to H^2(\mathbb{T})$, and $M_\phi$ is the bounded multiplication operator from $H^2(\mathbb{T}) \to L^2(\mathbb{T})$ given by $M_\phi f = \phi f$. Finally the converse is true, the bounded operator given by $T_\phi = P_{H^2(\mathbb{T})} M_\phi$ with $\phi \in L^\infty(\mathbb{T})$ satisfies $S^* T S = T$.

When the bounded condition is relaxed to closed and densely defined, the corresponding definitions of Toeplitz operators are no longer equivalent. For instance, if the coefficients of an upper triangular matrix are the coefficients of a Smirnov class function, $\phi \in N^+$, then the operator defined by the closure of this matrix, call it $T$, is densely defined, and the operator is the adjoint of a densely defined multiplication operator (an analytic Toeplitz operator) $M_\phi$. Unlike its bounded counterpart, $T$ can not be represented by a multiplication operator as $P M_{\bar \phi}$ since its domain is strictly larger than the domain of $M_{\bar\phi}$. The operator $T$ does satisfy the following algebraic equations:
\begin{enumerate}
\item $D(T)$ is $S$-invariant,
\item $S^* T S = T$, and
\item If $f \in D(T)$ and $f(0)=0$, then $S^*f \in D(T)$.
\end{enumerate}

These can be seen as the densely defined analogue of the algebraic condition for bounded Toeplitz operators. Therefore $T$ satisfies the algebraic conditions for being a Toeplitz operator, but is not a Toeplitz operator in the multiplication sense. However, $T$ is a closed extension of a multiplication type Toeptliz operator. At the close of \cite{sarason2008} the following problem was posed:

\begin{problem}Is it possible to characterize those closed densely defined operators $T$ on $H^2(\mathbb{T})$ with the above three properties? Moreover, is every closed densely defined operator on $H^2(\mathbb{T})$ that satisfies these conditions determined in some sense by a symbol?\end{problem}

This paper aims to address the second half of this question. If a closed densely defined operator, $T$, satisfies the three algebraic conditions above, henceforth a \emph{Sarason-Toeptliz} operator, then is $T$ the extension of an operator of the form $P M_\phi$ where $M_\phi$ is a densely defined multiplication operator from $H^2$ to $L^2$?

For bounded Toeplitz operators the recovery of the symbol of a Toeplitz operator can be achieved through the symbol map on $\mathcal{T}$, the algebra of generated by the collection of Toeplitz operators in $\mathcal{L}(H^2)$. Douglas demonstrated that there is a unique multiplicative mapping, $\phi$ from $\mathcal{T}$ to $L^\infty$ such that $\phi(T_f T_g) = \phi(T_f)\phi(T_g) = fg$ \cite{Douglas,axler}. This fact was proven again in \cite{barriahalmos} by Halmos and Barria using the limits along the diagonals of a Toeplitz matrix in order to find the symbol in $L^\infty$.

The Hardy space can be identified with analytic functions of the disc $\mathbb{D}$ such that the Taylor coefficients of these functions are square summable. By this viewpoint, $H^2$ is a reproducing kernel Hilbert space (RKHS) over $\mathbb{D}$ with the kernel functions $k_w(z) = (1-\bar w z)^{-1}$ for $|w| < 1$.

In the case of bounded Toeplitz operators, the Berezin transform, a tool particular to the study of RKHSs, is sufficient for the recovery of the of $L^\infty$ functions via radial limits of the Berezin transform of a bounded Toeplitz operator \cite{englis1995}. However, in more general cases the recovery of the symbol of a Sarason-Toeplitz operator is no longer clear.  The recovery of the symbol of a densely defined analytic (or a co-analytic) Toeplitz operator with symbol $\phi$ can be accomplished by the use of the Berezin transform. In this case, the adjoint of an analytic Toeplitz operator has the reproducing kernels as eigenvectors, $k_z$, with eigenvalues $\overline{\phi(z)}$ \cite{sarason2008}. Thus $$\tilde T(z) = (1-|z|^2)\langle k_z,T^* k_z\rangle = (1-|z|^2) \langle k_z, \overline{\phi(z)}k_z \rangle = \phi(z).$$The application of the Berezin transform requires the kernel functions $k_w(z) = (1-\bar w z)^{-1}$ to be in the domain of a operator or in the domain of its adjoint. Thus, the investigation of a new method is justified for the recovery of the symbol of a densely defined Sarason-Toeplitz operator.

We introduce the Sarason Sub-Symbol, which depends on a choice of a function in $D(T)$, as a family of symbol maps for Sarason-Toeplitz operators. In the development, it will be demonstrated that for the bounded case the Sarason Sub-Symbol is unique iff the operator is Toeplitz. Thus the uniqueness of the Sarason Sub-Symbol provides another equivalent definition for a bounded Toeplitz operator. Subsequently it is demonstrated that the Sarason Sub-Symbol for an analytic Toeplitz operator is unique and determines the operator. The rest of the paper is concerned with classes of Toeplitz operators for which the existence of the Sarason Sub-Symbol can be established, and it demonstrates sufficient conditions to show that $T$ is a closed extension of a multiplication type Toeplitz operator.

\section{The Problem of Sarason}

For bounded Toeplitz operators, the Sarason Problem has been long settled \cite{hw,Douglas}. Indeed, if a Toeplitz operator is bounded, then it can be represented by an $L^\infty$ function. Suarez characterized all closed densely defined operators on $H^2(\mathbb{T})$ that commute with the adjoint of the shift operator \cite{Suarez}, and Sarason gives a different treatment of operators that commute with the shift operator, the so called analytic Toeplitz operators. Both of these collections of operators satisfy the Sarason-Toeplitz condition. In addition, the analytic Toeplitz operators are precisely the operators of multiplication by an function in the Smirnov class, $N^+$ \cite{sarason2008}. Suarez's operators are the adjoints of these analytic Toeplitz operators and are called co-analytic Toeplitz operators \cite{Suarez}. Thus the above classes of Sarason-Toeplitz operators are completely characterized by a symbol.

Analytic and co-analytic Toeplitz operators both satisfy the Sarason conditions. The following propery generalizes this relationship.

\begin{proposition}If $T$ is a Sarason-Toeplitz operator then so is $T^*$.\end{proposition}

\begin{proof}$T$ is a closed densely defined operator, which means that $T^*$ is closed and densely defined as well. Thus $D(T^*)$ is nonempty.

To demonstrate that $T^*$ has a shift invariant domain, take $g\in D(T^*)$. By definition this means that $\tilde L(f) = \langle Tf, g \rangle$ is a continuous functional. In order to show that $zg \in D(T^*)$ it must be established that $L(f) = \langle Tf, zg \rangle$ is continuous. Note that $zD(T) \subset D(T)$, and $zD(T)$ has co-dimension 1 in $D(T)$. Thus there exists an $f_0 \in D(T)$ such that $$D(T) = \mathbb{C} \{ f_0 \} \oplus zD(T).$$ The functional $L$ is continouous on $\mathbb{C}\{f_0\}$, since it is finite dimensional. Therefore it suffices to show that $L$ is continuous on $zD(T)$. If $f = zh$ for some $h \in D(T)$, then $$L(f) = L(zh) = \langle Tzh, zg \rangle = \langle Th, g \rangle = \tilde L(h).$$ Thus $L$ is continous on $zD(T)$, since $\tilde L$ is continuous on $D(T)$.

Now suppose that $g \in D(T^*)$ and $g(0) = 0$, and consider the functional $L_2(f) = \langle Tf, S^*g \rangle$ defined for $f \in D(T)$. This functional can be rewritten as $$L_2(f) = \langle S^*TSf, S^*g \rangle = \langle TSf, g \rangle := \tilde L_2(Sf).$$ It follows that $L_2(f)$ is continuous, since $\tilde L_2(Sf)$ is continuous with respect to $f$.

Finally for all $f \in D(T^*)$ and $g \in D(T)$ we have,
$$\langle T^*f, g \rangle = \langle f, Tg \rangle = \langle f, S^* T S g \rangle = \langle S^*T^*S f, g \rangle,$$ which yields the second condition. 
\end{proof}

\section{The Sarason Sub-Symbol}

While the Berezin transform can be applied to recover the symbol of densely defined analytic and co-analytic Toeplitz operators, it is not clear if it can be used to recover the symbol of more general densely defined Toeplitz operators. This is because the functions $k_z$ a required to be in the domain of either the operator or the adjoint of the operator for the Berezin transform to be well defined. Instead, in this section the Sarason Sub-Symbol will be introduced as a candidate for the recovery of the symbol of densely defined Sarason-Toeplitz operators.

As a motivating example for the defintion of the Sarason Sub-Symbol, first suppose that $T$ is a bounded Toeplitz operator with symbol $\phi \in L^\infty$. In this case $$a_n = \left\{ \begin{array}{cc}\langle T1, z^n \rangle & n \ge 0\\ \langle Tz^n, 1 \rangle & n < 0 \end{array}\right.$$ are the Fourier coefficients of $\phi$. Thus $\phi$ can be reconstructed as follows $$\phi(e^{i\theta}) = \sum_{n=1}^\infty \langle Tz^n, 1 \rangle e^{-in\theta} + \sum_{n=0}^\infty \langle T1, z^n \rangle e^{in\theta}.$$

While it is not expected that $1 \in D(T)$ in general, given any function $f\in D(T)$ the domain of the densely defined operator $TM_f$ contains the polynomials, since $D(T)$ is shift invariant. The Sarason Sub-Symbol is defined as follows:
\begin{definition}Let $T$ be an operator with a shift invariant domain $D(T)$. For $f \in D(T)\setminus \{0\}$ the Sarason Sub-Symbol corresponding to $f$ is given by $R_f = h_f/f$ where $$h_f = \sum_{n=1}^\infty \langle Tfz^n, 1 \rangle e^{-in\theta} + \sum_{n=0}^\infty \langle Tf, z^n \rangle e^{in\theta}$$ where this series is convergent in some sense. The partial Sarason Sub-Symbol corresponding to $f$ is given by $R_{f,N} = h_{f,N}/f$ where $$h_{f,N} = \sum_{n=1}^N \langle Tfz^n, 1 \rangle e^{-in\theta} + \sum_{n=0}^\infty \langle Tf, z^n \rangle e^{in\theta}.$$\end{definition}

Heuristically, if $T$ is a Toeplitz operator associated with multiplication by the symbol $\phi$, then $h_f = \phi \cdot f$. The question of well definedness of the Sarason Sub-Symbol depends on the convergence of the series contained in the definition of $h_f$. When $\phi \in L^\infty$, $h_f = \phi \cdot f$, and is a well defined function in $L^2$. More specifically we can characterize all bounded Toeplitz operators by means of the Sarason Sub-Symbol.

\begin{proposition}\label{bddToep}Let $V$ be a bounded operator on $H^2$. The operator $V$ is a Toeplitz operator iff the Sarason Sub-Symbol is independent of the choice of $f \in H^2$.\end{proposition}

\begin{proof}Suppose that $V=T_\phi$ is a Toeplitz operator with symbol $\phi$. $$h_f = \sum_{n=1}^\infty \langle T_\phi f z^n, 1 \rangle_{H^2} e^{-in\theta} + \sum_{n=0}^\infty \langle T_\phi f, z^n \rangle_{H^2} e^{in\theta}$$
$$= \sum_{n=1}^\infty \langle \phi f z^n, 1 \rangle_{L^2} e^{-in\theta} + \sum_{n=0}^\infty \langle \phi f, z^n \rangle_{L^2} e^{in\theta} = \phi \cdot f.$$ Thus $R_f = h_f/f = \phi$ is independent of the choice of $f$.

Now suppose that $V$ is not a Toeplitz operator. This means there is a pair of integers $n,m \in \mathbb{N}$ such that $n < m$ (without loss of generality) and $\langle Vz^n,z^m \rangle \neq \langle V1, z^{m-n} \rangle$. In this case consider the two Sarason Sub-Symbols $$R_1 = \sum_{k=1}^\infty \langle Vz^k, 1\rangle e^{-ik\theta} + \sum_{k=0}^\infty \langle V1, z^k \rangle e^{ik\theta} = \sum_{k=-\infty}^\infty a_k e^{ik\theta}\text{ and}$$ $$R_{z^n}=e^{-in\theta} \left( \sum_{k=1}^\infty \langle Vz^{n+k}, 1\rangle e^{-ik\theta} + \sum_{k=0}^\infty \langle Vz^n, z^k \rangle e^{ik\theta}\right) = e^{-in\theta} \sum_{k=-\infty}^\infty b_k e^{ik\theta}.$$ The difference of the two sub-symbols yields $$R_1 - R_{z^n} = e^{-in\theta} \sum_{k=-\infty}^\infty(a_{k-n}-b_k)e^{ik\theta}.$$ The coeficient $(a_{m-n}-b_m) \neq 0$ by construction. Therefore $R_1 \neq R_{z^n}$. \end{proof}

Thus every bounded Toeplitz operator is characterized by the uniqueness of its Sarason Sub-Symbols. This motivates the investigation into densely defined operators. The following sections investigate the interplay between the Sarason Sub-Symbols and densely defined Sarason-Toeplitz operators.

\section{Analytic Densely Defined Toeplitz Operators}

Just as in Proposition \ref{bddToep}, an analytic densely defined Toeplitz operator is completely characterized by a symbol. As shown in \cite{sarason2008}, these operators are precisely the multiplication operators with symbols, $\phi$, in the Smirnov class of functions. That is, each $\phi$ can be written as a ratio of $H^\infty$ functions $b/a$ where $|a(e^{i\theta})|^2 + |b(e^{i\theta})|^2 = 1$ for all $\theta$  and $a$ an outer function. In this setting the Sarason Sub-Symbol is unique.

\begin{theorem}Given a Sarason-Toeplitz operator $T$, there exists a symbol $\phi \in N^+$ for which $T=M_\phi$ iff $\langle Tzf, 1 \rangle = 0$ for all $f \in D(T)$. Moreover, the Sarason Sub-Symbol is unique.\end{theorem}

\begin{proof}The forward direction follows since $T_\phi = M_\phi$ for $\phi \in N^+$. This means $TS=ST$, and $\langle Tzf, 1 \rangle = \langle zTf, 1 \rangle = 0$ since $1 \in (zD(T))^\perp$.

In order to establish sufficiency, let $f_1 = \sum_{n=0}^\infty a_n z^n, f_2 = \sum_{n=0}^\infty b_n z^n \in D(T)\setminus \{0\}$. By hypothesis, $h_{f_i} = \sum_{n=0}^\infty \langle Tf_i 1, z^n \rangle z^n = Tf_i \in H^2$ for $i=1,2$.

In order to establish uniqueness of the symbol, $R_{f_1} = R_{f_2}$, consider the function $h_1 f_2 - h_2 f_1 \in L^1(\mathbb{T})$. The Fourier series of $h_1f_2$ and $h_2f_1$ can be computed through convolution. Hence,
$$h_1 f_2 = \sum_{n=0}^\infty \left( \sum_{k=0}^n \langle Tf_1, z^{n-k} \rangle b_k \right) z^n = \sum_{n=0}^\infty \left( \sum_{k=0}^n \langle Tz^k f_1, z^n \rangle b_k \right) z^n \text{, and }$$
$$h_2 f_1 = \sum_{n=0}^\infty  \left( \sum_{k=0}^n \langle Tz^k f_2, z^n \rangle a_k \right) z^n.$$ The second equality follows since $S^* T S f = Tf$, and $\langle TSf, 1 \rangle = 0$ implies that $TSf(0)=0$ and $S S^* T S f = T S f$. This leads to $$H := h_1f_2 - h_2f_1 = \sum_{n=0}^\infty \left( \sum_{k=0}^n \langle Tz^kf_1, z^{n} \rangle b_k - \sum_{k=0}^n \langle Tz^k f_2, z^n \rangle a_k \right)z^n.$$ In order to establish that each coefficient is in fact zero, consider, for arbtrary $n$, the coefficient of $z^n$:
$$\hat H(n) = \sum_{k=0}^n \langle Tz^kf_1, z^{n} \rangle b_k - \sum_{k=0}^n \langle Tz^k f_2, z^n \rangle a_k$$
$$= \left\langle T\left( f_1 \left[ \sum_{k=0}^n b_k z^k - f_2 \right] - f_2 \left[ \sum_{k=0}^n a_k z^k - f_1 \right] \right), z^n \right\rangle.$$
The $H^2$ function inside of $T$ is in fact in the domain of $T$ by the properties of  Sarason-Toeplitz operators, and this function has a zero of order greater than $n$ at zero. Denote by $z^{n+1}F_n$, the function in the argument of $T$. By our hypothesis, $$\hat H(n) = \langle Tz^{n+1} F_n, z^n \rangle = \langle TzF_n, 1 \rangle = 0.$$
Therefore $R_{f_1} = R_{f_2}$ for any choice of $f_1, f_2 \in D(T)\setminus \{ 0 \}$, so let $\phi = R_{f_1}$ be the proposed symbol for the Sarason-Toeplitz operator $T$. $h_f = Tf \in H^2$ for each $f \in D(T)$. Further, given any $z \in \mathbb{D}$ there exists $f_z \in D(T)$ such that $f(z) \neq 0$ (this follows from the density of $D(T)$ in $H^2$). Thus $\phi = Tf_z/f_z$ is analytic at $z$ for every point $z \in \mathbb{D}$. Finally note that for each $f \in D(T)$, $M_\phi f = \phi f = ( Tf / f ) f = Tf$. Thus $T=M_\phi$ is a densely defined multiplication operator with an analytic symbol. By \cite{sarason2008}, $\phi \in N^+$.
\end{proof}

\begin{corollary}A Sarason-Toeplitz operator $T$ on $H^2$ is analytic ($ST=TS$) iff $\langle Tzf, 1 \rangle =0$ for all $f \in D(T)$.\end{corollary}

\section{Symbols that are ratios of $L^2$ functions and $H^2$ functions}
In the case of an analytic densely defined Toeplitz operator with symbol $\phi$ (expressed as $\phi = b/a$ in canonical form) , the domain is given by $D(T)=aH^2$. This means that there is an outer function, in particular $a$, in the domain of $T$. Moreover, since $T=M_\phi$, it is clear that $h_a = \phi a \in H^2$. Therefore, the existence of an outer function $f \in D(T)$ for which $h_f$ is well defined is straighforward in the case of analytic Toeptliz operators.

When we consider a co-analytic Toeptliz operator of the form $M_\phi^*$, its domain is given by $\mathcal{H}(b)$, the de Branges-Rovnyak space corresponding to $b$. $\mathcal{H}(b)$ contains the space $aH^2$ as a subspace. Hence, it also has an outer function in its domain. In particular, if $f=a\cdot p$, where $p$ is a polynomial, then $h_f$ (corresponding to $M_\phi^*$) is in $L^2$. Since $a$ is an outer function, the collection of all such $f$ is dense in $H^2$. Therefore, the set $$D_2(T) = \{ f \in D(T) : h_f \in L^2 \}$$ is dense in $D(T) = D(M_\phi^*) = \mathcal{H}(b)$.

The answer to the question of the nonemptiness (as well as density) of the space $D_2(T)$ is unknown for general Sarason-Toeplitz operators. In this section, the applicability of the Sarason Sub-Symbol is extended to include functions of the form $B/A$ where $B \in L^2$ and $A$ is an $H^2$ outer function.

\begin{lemma}\label{cdd} Let $\phi$ be a function on the unit circle that can be written as the ratio of an $L^2$ function and an $H^2$ outer function. Let $$D(M_\phi) = \{ f \in H^2 : \phi \cdot f \in L^2\}.$$ The operator $M_\phi : D(M_\phi) \to L^2$ is a closed densely defined operator on $H^2$.\end{lemma}

\begin{proof}
Write $\phi = B/A$ where $B \in L^2$ and $A \in H^2$ is an outer function. Since $B \cdot p \in L^2$ for every polynomial $p(z)$, we see that $A \cdot p \in D(M_\phi)$ for every polynomial $p$. Therefore, $D(M_\phi)$ is dense in $H^2$ by the outer property of $A$.

Now suppose that $\{f_n\} \subset D(M_\phi)$ and $f_n \to f \in H^2$. Suppose further that $M_\phi f_n \to F \in L^2$. Since $f_n \to f$ in the $L^2$ norm, there exists a subsequence, $\{ f_{n_j}\}$, such that $f_{n_j} \to f$ almost everywhere. Since $A$ is an outer function, $A(e^{i\theta}) \neq 0$ for almost every $\theta$. Thus $\phi f_{n_j} \to \phi f$ almost everywhere.

The subsequence $\phi f_{n_j} \to F$ in $L^2$ and so there is a subsequence $\phi f_{n_{j_k}} \to F$ almost everywhere. However, this subsquence also converges almost everywhere to $\phi f$. Thus we may conclude that $\phi f = F$ almost everywhere, which completes the proof.
\end{proof}

\begin{theorem}\label{l2symbol}Let $T$ be a Sarason-Toeplitz operator. If there is an $H^2$ outer function $f \in D(T)$ such that $\sum_{n=1}^\infty \langle Tz^nf, 1 \rangle \bar z^n \in L^2$, then $T$ extends a closed densely defined operator of the form $T_\phi = PM_\phi$ where $\phi=R_f$ is the ratio of an $L^2$ function and an $H^2$ outer function. Moreover, $D_2(T)$ is a dense subset of $D(T)$.\end{theorem}

\begin{proof}Let $f$ be an $H^2$ outer function in $D(T)$, and let $h_f$ be the corresponding numerator of the Sarason Sub-symbol corresponding to $f$. Express $h_f = \sum_{n=-\infty}^\infty b_n z^n$. By the properties of Sarason-Toeplitz operators, $b_{n-m} = \langle Tfz^m, z^n \rangle$. Now consider the operator $T_{R_f} = P M_{R_f}$ which is closed and densely defined by Lemma \ref{cdd}.

Since the domain of $T$ is shift invariant, $f \cdot p \in D(T)$ for every polynomial $p$. Moreover, $h_f \cdot p \in L^2$ for every polynomial $p$. It follows that $D_2(T) \subset D(T)$ is dense in $H^2$ since $f$ is an outer function. Define the set $F := \{ f \cdot p : p \text{ is a polynomial } \} \subset D_2(T)$.

Let $p(z) = a_k z^k + \cdots + a_1 z + a_0$ be a polynomial of degree $k \in \mathbb{N}$. The product of $h(z)$ and $p(z)$ can be calculated as follows:
$$h(z) \cdot p(z) = \sum_{n=-\infty}^\infty \left( \sum_{m=0}^k b_{n-m}a_m \right) z^n = \sum_{n=-\infty}^\infty \left( \sum_{m=0}^k \langle  Ta_m z^m f, z^n \rangle \right) z^n$$
$$=\sum_{n=-\infty}^\infty \langle Tfp, z^n \rangle z^n = w(z) + T(fp)(z).$$ Where $w(z) \in \overline{H^2_0}$ since $h_f \in L^2$. In particular, this means $T_{R_f}(fp) = P(hp) = T(fp)$ for all polynomials $p$. Hence, $T$ agrees with $T_{R_f}$ on a dense domain, and $T$ extends $\left. T_{R_f} \right|_F.$ Finally, by Lemma \ref{cdd}, $\left. T_{R_f}\right|_F$ is closable, and $\left. T_{R_f} \right|_F \subset T$ implies that $\left. T_{R_f} \right|_F^{**} \subset T^{**} = T$. 
\end{proof}

The above theorem relies on the ability to find an outer function in $D_2(T)$. Once such a function is found, $T$ is shown to be a closed extension of the corresponding operator $PM_{R_f}$. When such a function does not exist, it can be shown that $T$ is the limit of multiplication type Toeplitz operator on a restricted domain.

\begin{proposition}Suppose $T$ is a Sarason-Toeplitz operator, let $f \in D(T)$, and define $F=\{f \cdot p : p \text{ is a polynomial } \}$. There exists a sequence of multiplication type Toeplitz operators, $T_{\phi_M} = P M_{\phi_M}$ such that $T_{\phi_M}$ converges to $T$ strongly on all of $F$. Moreover, these operators have a common dense domain.\end{proposition}

\begin{proof}Let $f \in D(T)$ and let $p(z) = a_k z^k + \cdots + a_1 z + a_0$ be a polynomial of degree $k$. Now, as in Theorem \ref{l2symbol}, consider the product $h_{f,N}(z)p(z)$:
$$h_{f,N}(z)\cdot p(z) = \sum_{n=-N}^\infty \left( \sum_{m=0}^{\min(k,n+N)} b_{n-m} a_m \right) z^n$$
$$= \sum_{n=-N}^{k-N-1} \langle Tf(a_0+\cdots+a_{n+N z^{n+N}}, z^n \rangle z^n + \sum_{n=k-N}^\infty \langle Tfp, z^n\rangle z^n.$$

Therefore, $$T_{R_{f,N}} (fp) (z) = P(h_Np) (z)$$ $$=\sum_{n=0}^{k-N-1} \langle Tf(a_0+\cdots+a_{n+N z^{n+N}}, z^n \rangle z^n + \sum_{n=\min(k-N,0)}^\infty \langle Tfp, z^n\rangle z^n.$$

The left sum is empty for large enough $N$, therefore $\{ T_{R_{f,N}} (fp) \}$ is constant for large enough $N$. This means $T_{R_{f,N}} (fp) \to T(fp)$ as $N \to \infty$.

In order to find a common domain for each of these Toeplitz operators, consider the inner-outer factorization $f = f_i f_o$. The functions, $\tilde h_N = h_{f,N}/f_i \in L^2$ since $f_i$ has modulus 1 on the circle. Thus $\tilde h_N f_o p \in L^2$ for all polynomials $p$. This implies that $$F_0 = \{ f_0 p : p \text{ is a polynomial } \} \subset D(T_{R_{f,N}})$$ for all $N \in \mathbb{N}$.
\end{proof}

\section{An Example of a Non-Sarason Toeplitz Operator}

This section is concerned with demonstrating that a densely defined Toeplitz matrix does not necessarily define a Sarason-Toeplitz operator. In particular, this section will extend an upper triangular Toeplitz matrix and demonstrate that the domain of the extension is not shift invariant. An upper triangular Toeplitz matrix is a matrix of the form $$\left( \begin{array}{cccc}\gamma_0 & \gamma_1 & \gamma_2 & \\0 & \gamma_0 & \gamma_1 & \cdots\\ 0 & 0 & \gamma_0 & \\ & \vdots & & \ddots \end{array}\right).$$ As an operator over $H^2$, this matrix has a natural dense domain, namely the polynomials. The density of the domain does not depend on the sequence $\{ \gamma_n \}_{n\in \mathbb{N}}$. Following Sarason \cite{sarason2008}, this operator may be extended as $$Tf = \sum_{m=0}^\infty \left( \sum_{n=0}^\infty \gamma_n \hat f(n+m)\right) z^m$$ where the domain of $T$ is the collection of functions in $H^2$ for which $Tf \in H^2$.

\begin{theorem}\label{facttoep}Let $T$ be the extension of an upper triangular matrix given by $$Tf = \sum_{m=0}^\infty \left( \sum_{n=0}^\infty n! \hat f(n+m) \right) z^m.$$ The domain of $T$ is defined to be $D(T) = \{ f \in H^2 : Tf \in H^2 \}.$ Every function $f \in D(T)$ is an entire function and can be written as $f(z) = \sum_{n=0}^\infty a_n \frac{z^n}{n!}$ where $\sum_{n=0}^\infty a_n$ converges.\end{theorem}

\begin{lemma}\label{factlemm}The sequence $\{ c_m = \sum_{n=1}^\infty (n+1)^{-m} \}_{m=2}^\infty$ is an $l^2$ sequence.\end{lemma}
\begin{proof}[Proof of Lemma]Each term of the sequence can be bounded by $$\int_{0}^\infty \frac{1}{(1+x)^m} dx = \frac1{m-1}.$$ Thus, $c_m$ is bounded by an $l^2$ sequence, and so it is also $l^2$.\end{proof}

\begin{proof}[Proof of Theorem \ref{facttoep}]
First suppose that $f \in D(T)$. By definition, the zero-th coefficient of $Tf$ is given by $\sum_{n=0}^\infty n! \hat f(n)$, which must be a convergent series. Declaring $a_n = n! \hat f(n)$, it can be seen that $\sum a_n$ converges. Moreover, since $f(z) = \sum_{n=0}^\infty a_n \frac{z^n}{n!}$, the function $f$ must be an entire function.

For the other direction, suppose that $f(z) = \sum_{n=0}^\infty a_n \frac{z^n}{n!}$ where $\sum a_n$ converges. Define $d_0 = \sum_{n=0}^\infty n! \hat f(n) = \sum_{n=0}^\infty a_n$. Note that since $a_n$ converges so does $\sum_{n=0}^\infty a_n b_n$ for any positive monotonically decreasing sequence $\{ b_n \}$. Thus for each $m = 1, 2, ...$ the series $$d_m = \sum_{n=0}^\infty n! \hat f(n+m) = \sum_{n=0}^\infty \frac{a_{n+m}}{(n+1)(n+2)\cdots(n+m)}$$ converges. This enables us to define $Tf$ formally as $\sum_{m=0}^\infty d_m z^m$.

In order to demonstrate that $d_m$ is in $l^2$, write $d_m$ as follows: $$d_m = \frac{a_m}{m!} + \sum_{n=1}^\infty \frac{a_{n+m}}{(n+1)(n+2)\cdots(n+m)} := s_m + t_m.$$

The sequence $\{ s_m \} \in l^2$ since $a_m \to 0$. The sequence $\{ t_m \}$ is bounded by  $\sum_{n=1}^\infty (n+1)^m = c_m$ for sufficiently large $m$, since $|a_m| < 1$ for $m$ sufficiently large. By Lemma \ref{factlemm}, $t_m$ is in $l^2$. This completes the proof of the theorem.
\end{proof}

\begin{corollary}The domain of the operator given in Theorem \ref{facttoep} is not shift invariant.\end{corollary}

\begin{proof}The function $f(z) = \sum_{n=1}^\infty \frac{(-1)^n}{n} \frac{z^n}{n!} \in D(T)$, since $\sum_{n=1}^\infty (-1)^n/n$ converges. Now consider the function $$zf(z) = \sum_{n=1}^\infty \frac{(-1)^n}{n} \frac{z^{n+1}}{n!} = \sum_{n=2}^\infty \frac{(-1)^{n-1}}{n-1} \frac{z^{n}}{(n-1)!} = \sum_{n=2}^\infty \frac{(-1)^{n-1}n}{n-1} \frac{z^{n}}{n!}.$$ The series $\sum_{n=2}^\infty (-1)^{n-1} \frac{n}{n-1}$ does not converge, which means $zf(z) \not\in D(T)$.
\end{proof}

By applying the same techniques used in proving Theorem \ref{facttoep}, a slightly weaker result can be found when $n!$ is replaced by a sequence of complex numbers $\{ \gamma_n \}$ with the growth condition $|\gamma_{n+1}| > (n+1) |\gamma_n|$.

\begin{theorem}Let $\{\gamma_n\}$ be a sequence of complex numbers as described above, and define the operator $Tf = \sum_{m=0}^\infty \left( \sum_{n=0}^\infty \gamma_n \hat f (n+m) \right) z^m$ with the domain $D(T)=\{ f \in H^2 : Tf \in H^2 \}$. The operator $T$ is densely defined and functions of the form $f(z) = \sum_{n=0}^\infty a_n \frac{z^n}{\gamma_n}$ where $\sum |a_n| <\infty$ are in its domain.\end{theorem}

\section{Conclusion}

This paper aims to give a partial answer to a question posed by Donald Sarason in \cite{sarason2008}. At the end of his paper, Sarason asked if a closed densely defined operator satisfying certain algebraic properties analogous to that of a bounded Toeplitz operator is determined by a symbol in some sense. We call these operators Sarason-Toeplitz operators.

The Sarason sub-symbol was presented as a family of potential symbol maps for Sarason-Toeplitz operators. For bounded and analytic densely defined Toeplitz operators, the Sarason sub-symbol is unique and characterizes the operators. In the case of a coanalytic densely defined Toeplitz operator, the Sarason sub-symbol produces a densely defined multiplication type Toeplitz operator that agrees with the original coanalytic Toeplitz operator on a restricted domain. Finally, these results were extended to a broader class of Sarason-Toeplitz operators, provided their domains contain functions that are ratios of $L^2$ functions and $H^2$ outer funtions.

Section 6 focused on a densely defined coanalytic Toeplitz matrix. The domain of an extension of the matrix was completely classified, and it was shown that the densely defined operator is not of Sarason-Toeplitz type. This demonstrates that the definition of a Toeplitz matrix and that of a Sarason-Toeplitz operator do not coincide.

For general Sarason-Toeplitz operators $T$ the density of the set $D_2(T)$ is unknown. It is also unknown if this set contains any nontrivial elements. This is a question to be addressed in future research, and presumably, if it is to be proven true, the closedness of $T$ must be leveraged.

\bibliographystyle{elsarticle-num}
\bibliography{toeplitzbibliography}
\end{document}